\documentclass{amsart}

\usepackage{amsthm}
\usepackage{amssymb}
\usepackage[all]{xy}

\newtheorem{lemma}{Lemma}
\newtheorem{theorem}[lemma]{Theorem}
\newtheorem{corollary}[lemma]{Corollary}
\newtheorem{proposition}[lemma]{Proposition}

\numberwithin{equation}{section}

\newcommand*{\pd}{\mathop{\mathrm{pd}}\displaylimits}

\title[Generalization of the Pontryagin-Hill theorems]{A generalization of the Pontryagin-Hill theorems to projective modules over Pr\"{u}fer domains}
\author[J. E. Mac\'{\i}as-D\'{\i}az]{J. E. Mac\'{\i}as-D\'{\i}az} 
\address{Departamento de Matem\'{a}ticas y F\'{\i}sica, Universidad Aut\'{o}noma de Aguascalientes, Avenida Universidad 940, Ciudad Universitaria, Aguascalientes, Ags. 20100, Mexico}
\email{jemacias@correo.uaa.mx}

\date{\today}

\begin{document}

\begin{abstract}
Motivated by the Pontryagin-Hill criteria of freeness for abelian groups, we investigate conditions under which unions of ascending chains of projective modules are again projective. Several extensions of these criteria are proved for modules over arbitrary rings and domains, including a genuine generalization of Hill's theorem for projective modules over Pr\"{u}fer domains with a countable number of maximal ideals. More precisely, we prove that, over such domains, modules which are unions of countable ascending chains of projective, pure submodules are likewise projective. 
\end{abstract}

\maketitle

\section{Introduction}

In the last century, Lev Pontryagin and Paul Hill studied conditions under which torsion-free abelian groups are free. In their investigations, the concept of purity of subgroups was crucial. More precisely, a subgroup $H$ of an abelian group $G$ is \emph {pure} if every equation of the form $k x = a \in H$, with $k \in \mathbb {Z}$, is solvable in $H$ whenever it is solvable in $G$. Equivalently, solvability in $G$ of each system of equations of the form
\begin{equation}
\sum _{j = 1} ^m k _{i j} x _j = a _i \in H \quad (i = 1 , \dots , n),
\end{equation}
with every $k _{i j} \in \mathbb {Z}$, implies its solvability in $H$.

In $1934$, Pontryagin proved that a countable, torsion-free abelian group is free if and only if every finite rank, pure subgroup is free \cite {Pontryagin}. Equivalently, every properly ascending chain of pure subgroups of finite rank is finite. From the proof of this result, it follows that a torsion-free abelian group $G$ is free if there exists an ascending chain
\begin{equation}
0 = G _0 \hookrightarrow G _1 \hookrightarrow \dots \hookrightarrow G _n \hookrightarrow \dots \quad (n < \omega) \label{CountChain}
\end{equation}
consisting of pure subgroups of $G$ whose union is equal to $G$, such that every $G_n$ is free and countable.

Later, in $1970$, Hill established that, in order for an abelian group $G$ to be free, it is sufficient to prove that it is the union of a countable ascending chain \eqref {CountChain} of free, pure subgroups \cite {Hill}. In other words, Hill proved that the condition of countability on the cardinality of the links $G _n$ in Pontryagin's theorem was superfluous. The proof of this theorem relies on some important facts about commutative groups, one of them being that subgroups of torsion-free abelian groups can be embedded in pure subgroups of the same rank. Applications of these criteria may be actually found in a variety of algebraic results \cite{Cornelius, Eklof, Shelah, Mekler}.

In view of the importance of the Pontryagin-Hill theorems in algebra, it is highly desirable to explore the possibility to generalize these criteria to the more general scenario of projective modules over suitable rings. With that purpose in mind, the present work establishes positive answers in the search for extensions of these theorems. Section \ref{sec2} of this article is devoted to state some useful results concerning relative divisibility, purity, projectivity and localizations of modules. In section \ref{sec3}, we prove some propositions related to classes of families of sets which are relevant to our investigation, while the last section presents the generalizations of the Pontryagin-Hill criteria to projective modules over the domains of interest, namely, semi-hereditary domains. 

\section{Preliminaries\label{sec2}}

By a \emph {ring} we mean a ring with an identity element, and by an \emph {integral domain} or, simply, a \emph {domain}, we understand a commutative ring without divisors of zero. In this context, a \emph {Pr\"{u}fer domain} is a semi-hereditary domain, that is, a domain in which finitely generated ideals are projective.

Over an integral domain, any two maximal independent subsets of a torsion-free module have the same cardinality, the common cardinal number being called the \emph {rank} of the module. Clearly, if a torsion-free module over a domain has rank at most $\kappa$, for some cardinal number $\kappa$, then every submodule also has rank at most $\kappa$. Moreover, if a torsion-free module over an integral domain is at most $\kappa$-generated, then it has rank less than or equal to $\kappa$.

\subsection{Pure submodules}

Let $R$ be a ring. A submodule $N$ of an $R$-module $M$ is \emph {relatively divisible} if the inclusion $N \cap r M \hookrightarrow r N$ holds, for every $r \in R$. Equivalently, solvability in $M$ of equations of the form $r x = a \in N$, with $r \in R$, implies their solvability in $N$. We say that $N$ is \emph {pure} in $M$ if every finite system of equations
\begin{equation}
\sum _{j = 1} ^m r _{i j} x _j = a _i \in N \quad (i = 1 , \dots , n),
\end{equation}
with $r _{i j} \in R$, is solvable in $N$ whenever it is solvable in $M$. Under these circumstances, a short-exact sequence $0 \rightarrow N \rightarrow M \rightarrow Q \rightarrow 0$ is $RD$-exact (respectively, pure-exact) if $N$ is a relatively divisible (respectively, pure) submodule of $M$. Evidently, purity implies relative divisibility, and they both coincide for modules over Pr\"{u}fer domains \cite {Warfield}. Moreover, Pr\"{u}fer domains are the only integral domains for which relative divisibility and purity are equivalent \cite{Cartan-Eilenberg}.

The conditions of relative divisibility and purity posses many interesting properties. For instance, they are closed under unions of ascending chains of arbitrary lengths. Moreover, the intersection of relatively divisible submodules of a torsion-free module $M$ is again relatively divisible. Thus, for every subset $X$ of $M$, there exists a smallest relatively divisible submodule of $M$ containing $X$, called the \emph {relatively divisible hull} of $X$. For torsion-free modules over Pr\"{u}fer domains, this submodule of $M$ is called the \emph {purification} of $X$ in $M$, in view that it coincides with the smallest pure submodule $\langle X \rangle _*$ of $M$ containing $X$. It is worth noticing that, over integral domains, the relatively divisible hull of a submodule has the same rank as the submodule itself. For proofs of these and more properties on relative divisibility and purity of modules, refer to Sections I.7 and I.8 of \cite {Fuchs-Salce2}.

\subsection{Projective modules}

Some elementary results on projective modules will be needed in the present work. For the sake of simplicity, modules are defined over an integral domain $R$, unless stated otherwise.



The following result for modules over domains is of utmost importance. It was generalized by Endo \cite {Endo} to more general settings. Indeed, a more general result states that finitely presented, flat modules over arbitrary rings are projective.

\begin{theorem}[Cartier \cite {Cartier}]
Finitely generated, flat modules over integral domains are projective. \qed \label{Thm:2.3}
\end{theorem}

Let $B$ be a flat module. Our next result uses the fact that an exact sequence $0 \rightarrow A \rightarrow B \rightarrow C \rightarrow 0$ is pure-exact if and only if $C$ is flat (see \cite {Fuchs-Salce2}, Lemma VI.9.1).

\begin{lemma}
\label{Lemma:2.4}
Every finite rank, pure submodule of a projective module over an integral domain is a finitely generated, projective module.
\end{lemma}

\begin{proof}
Let $A$ be a finite rank, pure submodule of a projective module $M$. Without loss of generality, we can assume that $M$ is a free module and, moreover, that it is of finite rank. This means that $M$ is finitely generated. Now, the purity of $A$ implies that $M / A$ is a finitely generated, flat module. By Theorem \ref {Thm:2.3}, $A$ is a direct summand of $M$ and, hence, it is a finitely generated, projective module.
\end{proof}

\begin{lemma}
Every countable rank, pure submodule of a projective module over an integral domain is contained in a countably generated submodule. \label {Lemma:2.5}
\end{lemma}

\begin{proof}
Let $A$ be a countable rank, pure submodule of the projective module $M$. Then, there exists a free module $F$ which contains $M$ as a direct summand. So, $A$ is pure in $F$ and is contained in a countably generated, free summand $N$ of $F$. Then, the image of $N$ under the projection homomorphism of $F$ onto $M$ is again countably generated and contains $A$.
\end{proof}

\begin{corollary}
Over integral domains, projective modules of countable rank are countably generated. \label {Coro:2.5} \qed
\end{corollary}

\begin{corollary}
Every countably generated module of projective dimension at most $1$ over an integral domain is countably presented. \label {Coro:2.6} \qed
\end{corollary}

\subsection{Localizations}

Throughout this section, $R$ will represent an integral domain. Here, we obey the tradition of representing the localization of a torsion-free module $M$ at a multiplicatively closed set $S \subseteq R$ by $M _S$, and the localization of $M$ at the complement of a prime ideal $P$ of $R$ by $M _P$. Moreover, we denote by $D (M)$ the localization of $M$ with respect to $R \setminus 0$, and the set of maximal ideals of $R$ by $\max R$.

\begin{lemma}
Let $R$ have countably many maximal ideals. A torsion-free $R$-module $M$ is countably generated if and only if the localization of $M$ at every maximal ideal $P$ of $R$ is a countably generated $R _P$-module. \label{Lemma:2.12}
\end{lemma}

\begin{proof}
We only need to prove that the given condition is sufficient. For each maximal ideal $P$ of $R$, choose a countable generating set of the $R _P$-module $M _P$, consisting of elements of the form $a _P ^n / 1$, with $a _P ^n \in M$, for every $n < \omega$. Let $M ^\prime$ be the $R$-submodule of $D (M)$ generated by
\begin{equation}
Y = \bigcup _{P \in \max R} \{ a _P ^n / 1 : n < \omega \}.
\end{equation} 
Obviously, $Y$ is contained in $M$ and, therefore, $M ^\prime$ is contained in $M$ as a submodule. Moreover, for every maximal ideal $P$ of $R$, we have that $M _P = M _P ^\prime$ as $R _P$-modules; consequently, $M = M ^\prime$.
\end{proof}

\begin{proposition}
Let $R$ be any integral domain which is not a field, and let $S$ be a countably infinite, multiplicatively closed subset of $R$ which does not contain any unit of $R$. Then, there exists a countable ascending chain
\begin{equation}
0 = M _0 \hookrightarrow M _1 \hookrightarrow \dots \hookrightarrow M _n \hookrightarrow \dots \quad (n < \omega) \label {Eq:Chain2.6}
\end{equation}
of submodules of the $R$-module $J = R _S$, such that:
\begin{enumerate}
\item[(i)] every $M _n$ is projective,
\item[(ii)] none of the $M _n$ is pure in $J$, and
\item[(iii)] $J = \bigcup _{n < \omega} M _n$.
\end{enumerate} \label {Prop:2.13}
\end{proposition}

\begin{proof}
Let $S = \{s _n : n \in \mathbb {Z} ^+\}$. For every positive integer $n$, let $M _n$ be the $R$-submodule of $J$ given by
\begin{equation}
M _n = \left\langle \frac {1} {s _1 \cdot \ldots \cdot s _n} \right\rangle \cong R.
\end{equation}
Obviously, the sequence $\{ M _n \} _{n < \omega}$ forms an ascending chain \eqref {Eq:Chain2.6} of $R$-modules satisfying (i) and (iii). In order to establish (ii), it suffices to prove that every $M _n$ is not relatively divisible in $M$. Indeed, observe that the equation
\begin{equation}
s _{n + 1} x = \frac {1} {s _1 \cdot \ldots \cdot s _n} \label {Eq:RelDivis}
\end{equation}
is solvable in the torsion-free $R$-module $J$, the only solution being
\begin{equation}
x = \frac {1} {s _1 \cdot \ldots \cdot s _n s _{n + 1}}.
\end{equation}
However, \eqref {Eq:RelDivis} is not solvable in $M _n$, so that $M _n$ is not pure in $M$.
\end{proof}

The existence of the set $S$ in Proposition \ref {Prop:2.13} is guaranteed for every integral domain $R$ which is not a field. For instance, $S$ can be the set of all positive integer powers of a nonzero, non-invertible element of $R$.

\begin{proposition}
Let $R$ be an integral domain which is not a field, and let $S$ be a countably infinite, multiplicatively closed subset of $R$ which does not contain any unit of $R$. Then, there exists a pure-exact sequence
\begin{equation}
0 \rightarrow H \rightarrow F \rightarrow J \rightarrow 0,
\end{equation}
with $H$ and $F$ free $R$-modules of countably infinite rank, and $J = R _S$. \label {Prop:2.14}
\end{proposition}

\begin{proof}
Proposition \ref {Prop:2.13} guarantees the existence of a countable ascending chain \eqref {Eq:Chain2.6} of projective $R$-modules whose union equals $J$, and it follows by Auslander's lemma \cite {Auslander} that $\pd _R J \leq 1$. However, $J$ is an infinitely generated localization of $R$, so it cannot be projective. As a consequence, $J$ is a flat module of projective dimension $1$; moreover, since it is countably generated, then $J$ is countably presented by Corollary \ref {Coro:2.6}. Thus, there exists a pure-exact sequence
\begin{equation}
0 \rightarrow P \rightarrow F \rightarrow J \rightarrow 0
\end{equation}
of $R$-modules, with $P$ a countably generated, projective module, and $F$ a free module of countably infinite rank. By a well-known result of Eilenberg (exercise 1.1 in Chapter VI of \cite{Fuchs-Salce2}), there exists a free module $H$ of countably infinite rank, such that $P \oplus H \cong H$. Therefore, the induced sequence
\begin{equation}
0 \rightarrow P \oplus H \rightarrow F \oplus H \rightarrow J \rightarrow 0
\end{equation}
is pure-exact. Notice finally that $F \oplus H$ is isomorphic to $F$, whence the conclusion follows.
\end{proof}

It is important to keep in mind that the $R$-module $J$ in Proposition \ref {Prop:2.14} is not projective. Also, it is useful to point out that localizations of Pr\"{u}fer domains are again Pr\"{u}fer domains; particularly, localizations of Pr\"{u}fer domains at prime ideals are valuation domains \cite {Fuchs-Salce2}. 

\section{Families of modules\label{sec3}}

A \emph {$G (\aleph _0)$-family} of a module $M$ over a ring $R$ is a family $\mathcal {B}$ consisting of submodules of $M$, with the following properties:
\begin{enumerate}
\item[(i)] $0 , M \in \mathcal {B}$,
\item[(ii)] $\mathcal {B}$ is closed under unions of ascending chains of arbitrary lengths, and
\item[(iii)] for every $A _0 \in \mathcal {B}$ and every countable set $H \subseteq M$, there exists $A \in \mathcal {B}$ containing $A _0$ and $H$, such that $A / A _0$ is countably generated.
\end{enumerate}
Clearly, an intersection of a countable number of $G (\aleph _0)$-families of submodules of $M$ is again a $G (\aleph _0)$-family of submodules of $M$. Examples of $G (\aleph _0)$-families are axiom-$3$ families. By an \emph {axiom-$3$} family of $M$ we mean a family $\mathcal {B}$ of submodules of $M$ satisfying (i) and (iii) above, plus the property:
\begin{enumerate}
\item[(ii)$^\prime$] $\mathcal {B}$ is closed under arbitrary sums.
\end{enumerate}

A $G (\aleph _0)$-family of submodules of $M$ is a \emph {tight system} if, in addition, it satisfies:
\begin{enumerate}
\item[(iv)] for every $A \in \mathcal {B}$, $\pd _R A \leq 1$ and $\pd _R (M / A) \leq 1$.
\end{enumerate}

It is worth  mentioning that every module has a $G (\aleph _0)$-family of submodules, namely, the collection of all its submodules. However, the existence of a $G (\aleph _0)$-family of pure submodules is not guaranteed in general for every $R$-module, not even when $R$ is an integral domain. However, over a valuation domain, a torsion-free module of projective dimension less than or equal to $1$ has a tight system of pure submodules \cite {Bazzoni-Fuchs}.

In the proof of the following result, we use ideas previously applied in the investigation of the freeness of abelian groups \cite{Hill2, Hill3}, and Butler groups of infinite rank \cite {Bican, Vinsonhaler}.

\begin{theorem}
Let $R$ be a Pr\"{u}fer domain with a countable number of maximal ideals. Every torsion-free $R$-module of projective dimension at most equal to $1$ has a $G (\aleph _0)$-family of pure submodules. \label{G-family}
\end{theorem}

\begin{proof}
Let $\{ P _n \} _{n \in \mathbb {Z} ^+}$ be all the maximal ideals of $R$, and let $M$ be a torsion-free $R$-module of projective dimension at most equal to $1$. For every positive integer $n$, the module $M _{P _n}$ is torsion-free of projective dimension at most equal to $1$ over the valuation domain $R _{P _n}$ and, so, it has a $G (\aleph _0)$-family $\mathcal {B} _n$ consisting of pure $R _{P _n}$-submodules. Clearly, every $\mathcal {B} _n$ is also a family of pure submodules of $M _{P _n}$ when considered as an $R$-module. 

Let
\begin{equation}
\mathcal {B} _n ^\prime = \mathcal {B} _n \cap M = \{ A \cap M : A \in \mathcal {B} _n \} \quad (n \in \mathbb {Z} ^+).
\end{equation}
Obviously, every $\mathcal {B} _n ^\prime$ is a family of pure $R$-submodules of $M$ which is closed under unions of ascending chains, such that $0 , M \in \mathcal {B} _n ^\prime$. Therefore, the intersection $\mathcal {B}$ of all families $\mathcal {B} _n ^\prime$ is a family of pure $R$-submodules of $M$ satisfying properties (i) and (ii) of a $G (\aleph _0)$-family, and we only need to prove (iii).

Let $B _0 \in \mathcal {B}$, and let $H _0 \subseteq M$ be a countable set. For every positive integer $n$, there exists $A _n ^0 \in \mathcal {B} _n$ such that $B _0 = A _n ^0 \cap M$. Our argument hinges on the construction of a countable ascending chain
\begin{equation}
B _0 \hookrightarrow B _1 \hookrightarrow \dots \hookrightarrow B _m \hookrightarrow \dots \quad (m < \omega) \label {Eq:ChainEc}
\end{equation}
of submodules of the $R$-module $M$, such that:
\begin{enumerate}
\item[(a)] $H _0$ is contained in $B _1$, and
\item[(b)] $B _{m + 1}$ has countable rank over $B _m$, for every $m < \omega$.
\end{enumerate}
More precisely, for every $m < \omega$, $B _{m + 1}$ is the union of a countable ascending chain
\begin{equation}
B _m = B _0 ^{m + 1} \hookrightarrow B _1 ^{m + 1} \hookrightarrow \dots \hookrightarrow B _k ^{m + 1} \hookrightarrow \dots \quad (k < \omega) \label {Eq:NewChain3.3}
\end{equation}
of pure $R$-submodules of $M$, satisfying the following properties, for every $k < \omega$:
\begin{enumerate}
\item[(c$_m$)] $B _k ^{m + 1} = A _k ^{m + 1} \cap M$ with $A _k ^{m + 1} \in \mathcal {B} _k$,
\item[(d$_m$)] $A _k ^{m + 1}$ is a countably generated $R _{P _k}$-module over $A _k ^0$, and
\item[(e$_m$)] $B _k ^{m + 1}$ has countable rank over $B _0$.
\end{enumerate}

Assume that the links of the finite ascending chain 
\begin{equation}
B _0 \hookrightarrow B _1 \hookrightarrow \dots \hookrightarrow B _m
\end{equation}
have been constructed as desired, for some $m < \omega$. Moreover, let $n < \omega$ and assume that the links of countable ascending chain \eqref {Eq:NewChain3.3} have also been constructed as required, for every $k \leq n$. Choose a complete set of representatives $G _{n + 1} ^m$ of a countable generating system of the $R _{P _{n + 1}}$-module $A _{n + 1} ^m$ modulo $A _{n + 1} ^0$, and a complete set of representatives $H _n ^{m + 1}$ of a maximal independent system of the $R$-module $B _n ^{m + 1}$ modulo $B _0$. Then, there exists $A _{n + 1} ^{m + 1} \in \mathcal {B} _{n + 1}$ containing both $A _{n + 1} ^0$ and $G _{n + 1} ^m \cup H _n ^{m + 1} \cup H _0$, such that $A _{n + 1} ^{m + 1}$ is a countably generated $R _{P _{n + 1}}$-module over $A _{n + 1} ^0$. (It is worth noticing here that this procedure ensures that $B _1 ^1$ will be a module of countable rank over $B _0$, which contains $H _0$.)

Let $B _{n + 1} ^{m + 1} = A _{n + 1} ^{m + 1} \cap M$. Properties (c$_m$) and (d$_m$) are then satisfied for $k = n + 1$. Moreover, since the $R$-module $A _{n + 1} ^{m + 1}$ has countable rank over $A _{n + 1} ^0$, then
\begin{equation}
\frac {B _{n + 1} ^{m + 1}} {B _0} \cong \frac {A _{n + 1} ^0 + (A _{n + 1} ^{m + 1} \cap M)} {A _{n + 1} ^0} \hookrightarrow \frac {A _{n + 1} ^{m + 1}} {A _{n + 1} ^0}
\end{equation}
is also a countable rank $R$-module, so that property (e$_m$) is also satisfied for $k = n + 1$. By induction, we construct the countable ascending chain \eqref {Eq:NewChain3.3} with the desired properties, and let $B _{m + 1}$ be the union of the links of such chain. 

Inductively, we construct a countable ascending chain \eqref {Eq:ChainEc} with the desired properties. Then, $B = \bigcup _{m < \omega} B _m$ is an $R$-submodule of $M$ which contains $B _0$ and $H _0$. Moreover, 
\begin{equation}
B = \bigcup _{m <  \omega} \bigcup _{k < \omega} (A _k ^m \cap M) = \left( \bigcup _{m < \omega} A _n ^m \right) \cap M \quad (n \in \mathbb {Z} ^+).
\end{equation}
The fact that $\bigcup _{m < \omega} A _n ^m \in \mathcal {B} _n$, for every index $n$, implies that $B \in \mathcal {B}$.

In order to prove that $B / B _0$ is a countably generated $R$-module, observe that exactness of the localization functors yields that
\begin{equation}
\left( \frac {B} {B _0} \right) _{P _n} = \frac {\displaystyle {\left[ \bigcup _{m < \omega} (A _n ^m \cap M) \right] _{P _n}}} {(A _n ^0 \cap M) _{P _n}} = \frac {\displaystyle {\bigcup _{m < \omega} A _n ^m}} {A _n ^0} \quad (n < \omega)
\end{equation}
is a countably generated $R _{P _n}$-module. Lemma \ref {Lemma:2.12} implies now that $B / B _0$ is countably generated. 
\end{proof}

For the rest of this section, we will assume that $M$ is a torsion-free module over a Pr\"{u}fer domain, for which there exists a countable ascending chain \eqref {Eq:Chain2.6} of submodules, such that the following properties are satisfied:
\begin{enumerate}
\item[\bf P$_1$] every $M _n$ is relatively divisible in $M$,
\item[\bf P$_2$] every $M _ n$ has a $G (\aleph _0)$-family $\mathcal {B} _n$ of relatively divisible submodules,
\item[\bf P$_3$] every factor module $M _{n + 1} / M _n$ has a $G (\aleph _0)$-family $\mathcal {C} _n$ of relatively divisible submodules, and
\item[\bf P$_4$] $M = \bigcup _{n < \omega} M _n$.
\end{enumerate}

\begin{lemma}
For every $n < \omega$, the collection
\begin{equation}
\mathcal {B} _n ^\prime = \left\{ A \in \mathcal {B} _n : \frac {(A + M _j) \cap M _{j + 1}} {M _j} \in \mathcal {C} _j \textrm {, for every } j < n \right\} \label {Eq:Bnp}
\end{equation}
is a $G (\aleph _0)$-family of relatively divisible submodules of $M _n$. \label {Lemma:3.12}
\end{lemma}

\begin{proof}
The collection $\mathcal {B} _n ^\prime$ clearly satisfies conditions (i) and (ii) of a $G (\aleph _0)$-family of submodules of $M _n$. So, let $A _0 \in \mathcal {B} _n ^\prime$, and let $H _0 \subseteq M _n$ be a countable set. We will construct an ascending chain
\begin{equation}
A _0 \hookrightarrow A _1 \hookrightarrow \dots \hookrightarrow A _m \hookrightarrow \dots \quad (m < \omega) \label{Eq:AscChain}
\end{equation}
of submodules of $M _n$, such that:
\begin{enumerate}
\item[(a)] $A _1$ contains $H _0$,
\item[(b)] $A _i \in \mathcal {B} _n$, for every $i < \omega$,
\item[(c)] $A _{i + 1} / A _i$ is countably generated, for every $i < \omega$, and
\item[(d)] for every $i < \omega$ and $j < n$, there exists a module $K _{i , j} / M _j \in \mathcal {C} _j$, such that 
\begin{equation}
\frac {(A _i + M _j) \cap M _{j + 1}} {M _j} \hookrightarrow \frac {K _{i , j}} {M _j} \hookrightarrow \frac {(A _{i + 1} + M _j) \cap M _{j + 1}} {M _j}. \label{Eq:Inclusion}
\end{equation}
\end{enumerate}

To start with, we let $K _{0 , j} = (A _0 + M _j) \cap M _{j + 1}$, for every $j < n$. In general, suppose that, for some $m < \omega$, the links of the finite chain
\begin{equation}
A _0 \hookrightarrow A _1 \hookrightarrow \dots \hookrightarrow A _m
\end{equation}
have been constructed as desired. Since $A _m / A _0$ is countably generated, then we can choose a countable set of representatives $X _j \subseteq M _{j + 1}$ of a maximal independent system of the purification of $(A _m + M _j) \cap M _{j + 1}$ modulo $(A _0 + M _j) \cap M _{j + 1}$, for every $j < n$. For each such $j$, there exists $K _{m , j} / M _j \in \mathcal {C} _j$ containing both the module $[(A _0 + M _j) \cap M _{j + 1}] / M _j$ and the quotient set $X _j$ modulo $M _j$, such that
\begin{equation}
\frac {K _{m , j} / M _j} {[(A _0 + M _j) \cap M _{j + 1}] / M _j} \cong \frac {K _{m , j}} {(A _0 + M _j) \cap M _{j + 1}}
\end{equation}
is countably generated. Thus, it is readily checked that the first inclusion of \eqref {Eq:Inclusion} is satisfied for $i = m$. 

For every $j < n$, let $H _j ^m \subseteq K _{m , j}$ be a complete set of representatives of a countable generating system of $K _{m , j}$ modulo $(A _ 0 + M _j) \cap M _{j + 1}$. Then, there exists $A _{m + 1} \in \mathcal {B} _n$ containing both $A _m$ and $H _0 \cup ( \bigcup _{j < n} H _j ^m )$, such that $A _{m + 1} / A _m$ is countably generated. Our construction guarantees that
\begin{equation}
\frac {K _{m , j}} {(A _0 + M _j) \cap M _{j + 1}} \hookrightarrow \frac {(A _{m + 1} + M _j) \cap M _{j + 1}} {(A _0 + M _j) \cap M _{j + 1}} \quad (j < n),
\end{equation}
so that the second inclusion of \eqref {Eq:Inclusion} is satisfied also for $i = m$.

Inductively, we construct an ascending chain \eqref {Eq:AscChain} of submodules of $M _n$, satisfying (a), (b), (c) and (d) above. Then, the module $A = \bigcup _{m < \omega} A _m$ is a member of $\mathcal {B} _n$ which contains $A _0$ and $H _0$, and is countably generated over $A _0$. Moreover, notice that property (d) implies that
\begin{equation}
\frac {(A + M _j) \cap M _{j + 1}} {M _j} = \bigcup _{i < \omega} \frac {(A _i + M _j) \cap M _{j + 1}} {M _j} = \bigcup _{i < \omega} \frac {K _{i , j}} {M _j} \in \mathcal {C} _j \quad (j < n).
\end{equation}
Consequently, $A \in \mathcal {B} _n ^\prime$, as desired.
\end{proof}

Next, we construct a $G (\aleph _0)$-family of relatively divisible submodules of $M$.

\begin{lemma}
The collection
\begin{equation}
\mathcal {B} = \{ A \hookrightarrow M : A \cap M _n \in \mathcal {B} _n ^\prime \text {, for every } n < \omega \} \label {Eq:B}
\end{equation}
is a $G (\aleph _0)$-family of relatively divisible submodules of $M$. \label {Lemma:3.13}
\end{lemma}

\begin{proof}
Again, conditions (i) and (ii) in the definition of a $G (\aleph _0)$-family are obvious. So, let $A _0 \in \mathcal {B}$, and let $H _0$ be a countable subset of $M$. We will construct a countable ascending chain
\begin{equation}
A _0 \hookrightarrow A _1 \hookrightarrow \dots \hookrightarrow A _n \hookrightarrow \dots \quad (n < \omega) \label{Eq:AscChain2}
\end{equation}
of submodules of $M$, such that:
\begin{enumerate}
\item[(a)] $A _1$ contains $H _0$, and
\item[(b)] the factor module $A _n / A _0$ is countably generated, for every $n < \omega$.
\end{enumerate}
More precisely, for every $n < \omega$, the module $A _n$ is the union of a countable ascending chain
\begin{equation}
0 = A _0 ^n \hookrightarrow A _1 ^n \hookrightarrow \dots \hookrightarrow A _k ^n \hookrightarrow \dots \quad (k < \omega) \label {Eq:Chainkn}
\end{equation}
of submodules of $M$, for which the following properties are satisfied:
\begin{enumerate}
\item[(c$_n$)] $A _k ^n \in \mathcal {B} _k ^\prime$, for every $k < \omega$,
\item[(d$_n$)] $A _k ^n$ is countably generated over $A _0 \cap M _k$, for every $k < \omega$, and 
\item[(e$_n$)] $A _k ^n \hookrightarrow A _n \cap M _k \hookrightarrow A _k ^{n + 1}$, for every $k < \omega$.
\end{enumerate}

Let $A _k ^0 = A _0 \cap M _k$, for every $k < \omega$. Obviously, $A _0$ is the union of chain \eqref {Eq:Chainkn} with $n = 0$, and properties (c$_0$) and (d$_0$) are satisfied, as well as the first inclusion of (e$_0$). So, let $m$ be now any nonnegative integer, and suppose that the links of chain \eqref {Eq:AscChain2} have been constructed as needed, for every $n \leq m$. Condition (b) implies that the module
\begin{equation}
\frac {A _m \cap M _k} {A _0 \cap M _k} \cong \frac {A _0 + (A _m \cap M _k)} {A _0} \hookrightarrow \frac {A _m} {A _0} \quad (k < \omega)
\end{equation}
has countable rank. Choose then a countable set $Y _k ^m \subseteq M _k$ of representatives of a maximal independent system of the purification of $(A _m \cap M _k) / (A _0 \cap M _k)$ in $M / (A _0 \cap M _k)$. Clearly, there exists $B _k ^m \in \mathcal {B} _k ^\prime$ containing both $A _0 \cap M _k$ and $Y _k ^m$, such that $B _k ^m$ is countably generated over $A _0 \cap M _k$; so, we may fix a complete set of representatives $H _k ^m \subseteq B _k ^m$ of a countable generating set of $B _k ^m$ modulo $A _0 \cap M _k$. Moreover, since $B _k ^m$ is relatively divisible in $M$, then
\begin{equation}
\left\langle \frac {A _m \cap M _k} {A _0 \cap M _k} \right\rangle _* \hookrightarrow \frac {B _k ^m} {A _0 \cap M _k} \hookrightarrow \frac {M _k} {A _0 \cap M _k} \quad (k < \omega).
\end{equation} 

Let $k$ be a nonnegative integer. Assume that the links of the finite chain
\begin{equation}
0 = A _0 ^{m + 1} \hookrightarrow A _1 ^{m + 1} \hookrightarrow \dots \hookrightarrow A _k ^{m + 1}
\end{equation}
have been constructed as required, and let $X _k \subseteq M _k$ be a complete set of representatives of a countable generating system of $A _k ^{m + 1}$ modulo $A _0 \cap M _k$. Lemma \ref {Lemma:3.12} implies that there exists $A _{k + 1} ^{m + 1} \in \mathcal {B} _{k + 1} ^\prime$ which contains $A _0 \cap M _{k + 1}$ and the countable set $X _k \cup H _{k + 1} ^m \cup (H _0 \cap M _{k + 1})$, such that $A _{k + 1} ^{m + 1}$ is countably generated over $A _0 \cap M _{k + 1}$. 

Inductively, we can construct a countable ascending chain \eqref {Eq:Chainkn}, with $n = m + 1$, and define the module $A _{m + 1}$ as the union of the links of \eqref {Eq:Chainkn}, in such way that conditions (c$_{m + 1}$), (d$_{m + 1}$) and (e$_{m + 1}$) be satisfied. Moreover, our construction guarantees that $A _m$ is contained in $A _{m + 1}$. Furthermore, the fact that $A _{m + 1} / A _0$ is countably generated follows from the isomorphism
\begin{equation}
\frac {A _k ^{m + 1}} {A _0 \cap M _k} \cong \frac {A _0 + (A _k ^{m + 1} \cap M _k)} {A _0} \quad (k < \omega). \label {Eq:Iso}
\end{equation} 
Indeed, the left-hand side of \eqref {Eq:Iso} is countably generated by properties (d$_{m + 1}$), so that the union over all indexes $k < \omega$ of the modules in the right-hand side is countably generated, too.

By induction, a countable ascending chain \eqref {Eq:AscChain2} satisfying properties (a) and (b) is constructed. The module $A = \bigcup _{n < \omega} A _n$ contains $A _0$ and $H _0$, and is countably generated over $A _0$. Moreover, properties (c$_n$) and (e$_n$) yield that
\begin{equation}
A \cap M _k = \bigcup _{n < \omega} (A _n \cap M _k) = \bigcup _{n < \omega} A _k ^n \in \mathcal {B} _k ^\prime \quad (k < \omega).
\end{equation}
Consequently, $A \in \mathcal {B}$.
\end{proof}

\begin{lemma}
For every $A \in \mathcal {B}$ and $n < \omega$, $A + M _n$ is relatively divisible in $M$. \label {Lemma:3.14}
\end{lemma}

\begin{proof}
It is enough to prove that $(A + M _n) \cap M _k$ is relatively divisible in $M$, for every $k > n$. First of all, observe that $A \cap M _{n + 1} \in \mathcal {B} _{n + 1} ^\prime$, whence it follows that $(A + M _n) \cap M _{n + 1}$ is likewise relatively divisible in $M$, for every $n < \omega$. So, our claim is true for $k = n + 1$.

Assume now that $(A + M _n) \cap M _k$ is relatively divisible in $M$, for some $k > n$. The modular law and the isomorphism theorem yield that
\begin{equation}
\frac {(A + M _k) \cap M _{k + 1}} {(A + M _n) \cap M _{k + 1}} \cong \frac {M _k} {(A + M _n) \cap M _k}. \label {Eq:ChainEqual}
\end{equation}
The last module in \eqref {Eq:ChainEqual} is torsion-free, so that the first module is also torsion-free. This implies that $(A + M _n) \cap M _{k + 1}$ is relatively divisible in $(A + M _k) \cap M _{k + 1}$. Finally, since $(A + M _k) \cap M _{k + 1}$ is relatively divisible in $M$, we conclude that $(A + M _n) \cap M _{k + 1}$ is also relatively divisible in $M$. Our claim follows now by induction.
\end{proof}

\section{A generalization of Hill's theorem}

In this section, we use the generalization of Pontryagin's criterion of freeness to projective modules, presented as Theorem 1.3, Chapter XVI in \cite {Fuchs-Salce2}, in order to provide a generalization of Hill's criterion of freeness. Particularly, we use the fact that a countable rank, torsion-free module over a Pr\"{u}fer domain is projective if and only if every finite rank, pure submodule is projective.

Beforehand, it is important to mention that the problem of generalizing Hill's criterion was attacked previously in \cite {Fuchs-Salce2} (see Theorem 1.4, Chapter XVI). However, the proof of that version of Hill's theorem for projective modules is wrong, one serious problem being that it is not generally true that the modules $U _1 \cap M _\nu$ in the proof of Lemma XVI.1.6 have rank less than or equal to $\kappa$.

The following is our generalization of Hill's theorem to projectivity of modules over Pr\"{u}fer domains. 

\begin{theorem}
A module $M$ over a Pr\"{u}fer domain is projective if there exists a countable ascending chain 
\begin{equation}
0 = M _0 \hookrightarrow M _1 \hookrightarrow \dots \hookrightarrow M _n \hookrightarrow \dots \quad (n < \omega) \label {Eq:MainChain}
\end{equation}
of submodules of $M$, such that:
\begin{enumerate}
\item[\rm (i)] every $M _n$ is projective,
\item[\rm (ii)] every $M _n$ is pure in $M$,
\item[\rm (iii)] every factor $M _{n + 1} / M _n$ admits a $G (\aleph _0)$-family $\mathcal {C} _n$ of pure submodules, and
\item[\rm (iv)] $M = \bigcup _{n < \omega} M _n$.
\end{enumerate} \label{Thm:9}
\end{theorem}

In the following discussion, assume the hypotheses of Theorem \ref {Thm:9}. By Kaplansly's theorem on decomposition of projective modules over Pr\"{u}fer domains, every module $M _n$ can be written as the direct sum of finitely generated modules $M _\alpha ^n$, for $\alpha$ in a set of indices $\Omega _n$. Clearly, the set
\begin{equation}
\mathcal {B} _n = \{A \hookrightarrow M _n : A = \oplus _{\alpha \in \Lambda} M _\alpha ^n \textrm{, for some } \Lambda \subseteq \Omega\}
\end{equation}
is an axiom-$3$ family of direct summands of $M _n$, for every $n < \omega$. By Lemma \ref {Lemma:3.12}, the collection $\mathcal {B} _n ^\prime$ given by \eqref {Eq:Bnp} is a $G (\aleph _0)$-family of relatively divisible submodules of $M _n$. Moreover, Lemma \ref {Lemma:3.13} and Lemma \ref {Lemma:3.14} state that the family $\mathcal {B}$ provided by \eqref {Eq:B} is a $G (\aleph _0)$-family of relatively divisible submodules of $M$, such that for every $A \in \mathcal {B}$ and every $n < \omega$, $A + M _n$ is relatively divisible in $M$.

\begin{lemma}
For every $A \in \mathcal {B}$, finite rank, pure submodules of $M / A$ are projective. \label {Lemma:10}
\end{lemma}

\begin{proof}
Let $D$ be a submodule of $M$ containing $A$, such that $D / A$ is a finite rank, pure submodule of $M / A$. Choose a maximal independent system $\{d _i + A : i = 1 , \dots , n\}$ of $D / A$, and let $S = \{ d _1 , \dots , d _n \} \subseteq D$. Let $k \in \mathbb {Z} ^+$ be such that $S \subseteq M _k$. Since $A + M _k$ is pure, $D \hookrightarrow A + M _k$. So, $D / A \hookrightarrow (A + M _k) / A \cong M _k / (A \cap M _k)$, which is projective. By Lemma \ref {Lemma:2.4}, $D / A$ is likewise projective.
\end{proof}

We are now in a position to provide a proof of our main result.

\begin{proof}[Proof of Theorem \ref {Thm:9}]
Let $\alpha$ be any nonzero ordinal, and let 
\begin{equation}
0 = A _0 \hookrightarrow A _1 \hookrightarrow \dots \hookrightarrow A _\gamma \hookrightarrow A _{\gamma + 1} \hookrightarrow \dots \quad (\gamma < \alpha)
\end{equation}
be an ascending chain of modules in $\mathcal {B}$, such that every factor module $A _{\gamma + 1} / A _\gamma$ is projective. If $\alpha$ is a limit ordinal, we let $A _\alpha = \bigcup _{\gamma < \alpha} A _\gamma$. Otherwise, if there exists $x \in M \setminus A _{\alpha - 1}$, we can pick a module $A _\alpha \in \mathcal {B}$ which contains $x$ and $A _{\alpha - 1}$, such that $A _\alpha / A _{\alpha - 1}$ has countable rank. By Lemma \ref {Lemma:10}, finite rank, pure submodules of the torsion-free module $A _\alpha / A _{\alpha - 1}$ are projective, so that $A _\alpha / A _{\alpha - 1}$ itself is projective. In such way, transfinite induction provides a continuous, well-ordered, ascending chain 
\begin{equation}
0 = A _0 \hookrightarrow A _1 \hookrightarrow \dots \hookrightarrow A _\alpha \hookrightarrow A _{\alpha + 1} \hookrightarrow \dots \quad (\alpha < \tau) \label {Eq:Chain:Lemma:1}
\end{equation}
of submodules of $M$, satisfying the hypotheses of Lemma XVI.1.1 in \cite {Fuchs-Salce2}. Consequently, $M$ is projective.
\end{proof}

\begin{corollary}
Let $R$ be a Pr\"{u}fer domain with a countable number of maximal ideals. An $R$-module $M$ is projective if there exists a countable ascending chain \eqref {Eq:MainChain} of projective, pure submodules of $M$, such that $M = \bigcup _{n < \omega} M _n$. \label{Coro:11}
\end{corollary}

\begin{proof}
For every $n < \omega$, $M _{n + 1}$ is torsion-free and contains $M _n$ as a pure submodule, so $M _{n + 1} / M _n$ is a torsion-free module of projective dimension at most $1$. Then, every factor module of \eqref {Eq:MainChain} admits a $G (\aleph _0)$-family of pure submodules by Theorem \ref {G-family}. By Theorem \ref {Thm:9}, $M$ is projective.
\end{proof}

Since valuation domains have a unique maximal ideal, Corollary \ref {Coro:11} can be obviously improved as follows.

\begin{corollary}
Let $M$ be a module over a valuation domain. If there exists a countable ascending chain \eqref {Eq:MainChain} of free, pure submodules of $M$ whose union is equal to $M$, then $M$ itself is free. \qed \label{Coro:12}
\end{corollary}

Let $\lambda$ be an infinite cardinal number. We say that a module is \emph {$\lambda$-free} if each of its submodules of rank less than $\lambda$ can be embedded in a free, pure submodule.

\begin{corollary}
Let $\lambda$ be an infinite cardinal with co-finality equal to $\omega$. Every torsion-free, $\lambda$-free module of rank $\lambda$ over a valuation domain is free.
\end{corollary}

\begin{proof}
Let $M$ be a torsion-free, $\lambda$-free module over a valuation domain, let $X = \{a _\alpha \in M : \alpha < \lambda\}$ be a maximal independent set in $M$, and let 
\begin{equation}
\lambda _0 < \lambda _1 < \dots < \lambda _n < \dots \quad (n < \omega)
\end{equation}
be a sequence of ordinals whose union is $\lambda$. For every $n < \omega$, the purification in $M$ of the set $\{a _\alpha \in M : \alpha < \lambda _n\}$ has rank less than $\lambda$ and, so, it is contained in a free, pure submodule $M _n$ of $M$. Obviously, the modules $\{ M _n \} _{n < \omega}$ may be chosen to form a countable ascending chain of submodules of $M$ whose union contains $X$. The conclusion of this result follows now from Corollary \ref {Coro:12}.
\end{proof}

As a closing remark, it is worth mentioning that, as of now, we still do not know whether it is possible to extend Theorem \ref {Thm:9} to projectivity of modules over integral domains in general.

\subsubsection*{Acknowledgment}
The author wishes to acknowledge the guidance and support of Prof. L\'{a}szl\'{o} Fuchs at practically every stage of this work. Likewise, he wishes to thank the anonymous referee for her/his invaluable comments, which led to a substantial improvement of the manuscript overall and, in particular, to elegant simplifications of the proofs of several of the results.


\begin{thebibliography}{10}
\bibitem{Auslander}
L.~Auslander.
\newblock {On the dimension of modules and algebras. III}.
\newblock {\em Nagoya Math. J.}, {\bf 9}:67--77, 1955.

\bibitem{Bazzoni-Fuchs}
S.~Bazzoni and L.~Fuchs.
\newblock {On modules of finite projective dimension over valuation domains}.
\newblock In {\em Proceedings of the Conference on Abelian Groups and Modules in Udine}, volume {287} of {\em {CISM Courses and Lectures}}, pages 361--371. Springer, 1984.

\bibitem{Bican}
L.~Bican and K.~M.~Rangaswamy.
\newblock {Smooth unions of Butler groups}.
\newblock {\em Forum Mathematicum}, {\bf 10}(2):233–-247, 1998.

\bibitem{Vinsonhaler}
L.~Bican,  K.~M.~Rangaswamy and C.~Vinsonhaler.
\newblock {Butler groups as smooth ascending unions}.
\newblock {\em Commun. in Algebra}, {\bf 28}(11):5039--5045, 2000.

\bibitem{Cartan-Eilenberg}
H.~Cartan and S.~Eilenberg.
\newblock {\em {Homological Algebra}}.
\newblock {Princeton Landmarks in Mathematics}. Princeton University Press, Princeton, New Jersey, 1st edition, 1999.

\bibitem{Cartier}
P.~Cartier.
\newblock {Questions de rationalit\'{e} de diviseurs en g\'{e}om\'{e}trie alg\'{e}brique}.
\newblock {\em Bull. Soc. Math. France}, {\bf 86}:177--251, 1958.

\bibitem{Cornelius}
E.~F. Cornelius.
\newblock {A generalization of separable groups}.
\newblock {\em Pac. J. Math.}, {\bf 39}(3):603--613, 1971.

\bibitem{Eklof}
P.~C. Eklof.
\newblock {Whitehead's problem is undecidable}.
\newblock {\em Amer. Math. Month.}, {\bf 83}(10):775--788, 1976.

\bibitem{Shelah}
P.~C. Eklof and S.~Shelah.
\newblock {A non-reflexive Whitehead group}.
\newblock {\em J. Pure Appl. Alg.}, {\bf 156}(2-3):199--214, 2001.

\bibitem{Endo}
S.~Endo.
\newblock {On flat modules over commutative rings}.
\newblock {\em J. Math. Soc. Japan}, {\bf 14}:284--291, 1962.

\bibitem{Fuchs-Salce}
L.~Fuchs and L.~Salce.
\newblock {\em {Modules over Valuation Domains}}, volume~91 of {\em {Lecture Notes in Pure and Applied Mathematics}}.
\newblock Marcel-Dekker, New York, 1st edition, 1985.

\bibitem{Fuchs-Salce2}
L.~Fuchs and L.~Salce.
\newblock {\em {Modules over non-Noetherian Domains}}, volume~84 of {\em {Mathematical Surveys and Monographs}}.
\newblock American Mathematical Society, Providence, Rhode Island, 1st edition, 2001.

\bibitem{Hill}
P.~Hill.
\newblock {On the freeness of abelian groups: a generalization of Pontryagin's theorem}.
\newblock {\em Bull. Amer. Math. Soc.}, {\bf 76}(5):1118--1120, 1970.

\bibitem{Hill2}
P.~Hill.
\newblock {New criteria for freeness in abelian groups}.
\newblock {\em Trans. Amer. Math. Soc.}, {\bf 182}:201--209, 1973.

\bibitem{Hill3}
P.~Hill.
\newblock {New criteria for freeness in abelian groups. II}.
\newblock {\em Trans. Amer. Math. Soc.}, {\bf 196}:191--201, 1974.

\bibitem{Jech}
T.~Jech.
\newblock {\em {Set Theory}}.
\newblock {Springer Monographs in Mathematics}. Springer-Verlag, Berlin, 3th edition, 2003.

\bibitem{Kaplansky}
I.~Kaplansky.
\newblock {Projective modules}.
\newblock {\em Annals of Math.}, {\bf 68}(2):372--377, 1958.

\bibitem{Mekler}
A.~H. Mekler and S.~Shelah.
\newblock {Every coseparable group may be free}.
\newblock {\em Israel J. Math.}, {\bf 81}(1-2):161--178, 1993.

\bibitem{Pontryagin}
L.~Pontryagin.
\newblock {The theory of topological commutative groups}.
\newblock {\em Annals of Math.}, {\bf 35}(2):361--388, 1934.

\bibitem{Warfield}
R.~B. Warfield.
\newblock {Purity and algebraic compactness of modules}.
\newblock {\em Pac. J. Math.}, {\bf 28}:699--719, 1969.
\end{thebibliography}
\end{document}